\documentclass[12pt]{amsart}
\usepackage{fullpage}
\usepackage[utf8]{inputenc}
\usepackage{amssymb}
\usepackage{amsmath,amsthm}
\usepackage{color}

\usepackage{hyperref}
\usepackage{mathtools}
\usepackage{tikz}
\usepackage{comment}
\usepackage{soul}
\usepackage{cancel}

\newtheorem{theorem}{Theorem}

\newtheorem{lemma}[theorem]{Lemma}
\newtheorem{corollary}[theorem]{Corollary}

\newtheorem{example}[theorem]{Example}

\newtheorem*{acknowledgement*}{Acknowledgement}

\title{Transcendence measure of $e^{1/n}$}
\author[M. Dujella]{Marta Dujella}
\address[Marta Dujella]{Universit\"at Basel, Spiegelgasse 1, 4051 Basel, Switzerland.}\email{marta.dujella@unibas.ch}
\author[A.-M. Ernvall-Hytonen]{Anne-Maria Ernvall-Hyt\"onen}
	\address[Anne-Maria Ernvall-Hyt\"onen]{University of Helsinki, PL 68, 00014 Helsingin yliopisto, Finland.}\email{anne-maria.ernvall-hytonen@helsinki.fi}
\author[L. Frey]{Linda Frey}
\address[Linda Frey]{Universit\"at G\"ottingen, Bunsenstra\ss e 3-5, 37073 G\"ottingen, Germany.}\email{lindafrey89@gmail.com}
\author[B. Roy]{Bidisha Roy}
\address[Bidisha Roy]{Scuola Normale di Pisa, Piazza dei Cavlieri,7, 56126, Pisa, Italy.}\email{brroy123456@gmail.com}

\thanks{We thank the organizers of the conference Women in Numbers Europe 4, especially but not exclusively Valentijn Karemaker and Nirvana Coppola. This conference set the foundation for this article. Furthermore, we thank the Universiteit Utrecht for granting unlimited coffee.}

\begin{document}

\maketitle
\begin{abstract}
   For a given transcendental number $\xi$ and for any polynomial $P(X)=: \lambda_0+\cdots+\lambda_k X^k \in \mathbb{Z}[X]$, we know that $ P(\xi) \neq 0.$  Let $k \geq 1$ and $\omega (k, H)$ be the infimum of the numbers $r > 0$ satisfying the estimate
$$
\left|\lambda_0+\lambda_1 \xi+\lambda_2 \xi^{2}+ \ldots +\lambda_k\xi^{k}\right| > \frac{1}{H^r},
$$
for all $(\lambda_0, \ldots ,\lambda_k)^T \in \mathbb{Z}^{k+1}\setminus\{\overline{0}\}$ with $\max_{1\le i\le k} \{|\lambda_i|\} \le H$.  Any function greater than or equal to $\omega (k,H)$ is a {\it transcendence measure of $\xi$}. In this article, we find out a transcendence measure of $ e^{1/n}$ which improves a result proved by Mahler(\cite{Mahler}) in 1975. 
\end{abstract}

\section{Introduction}

Let $\xi$ be a transcendental number. Then $P(\xi)\ne 0$ for any polynomial $P$ with integer coefficient.

Let $k, H \ge 1$ and $\omega (k, H)$ is the infimum of the numbers $r > 0$ satisfying the estimate
\begin{equation}\label{ensikaava}
\left|\lambda_0+\lambda_1 \xi+\lambda_2 \xi^{2}+ \ldots +\lambda_k\xi^{k}\right| > \frac{1}{H^r},
\end{equation}
for all $\overline{\lambda}=(\lambda_0, \ldots ,\lambda_k)^T \in \mathbb{Z}^{k+1}\setminus\{\overline{0}\}$ with $\max_{1\le i\le k} \{|\lambda_i|\} \le H$. Any function greater than or equal to $\omega (k,H)$ is a {\it transcendence measure of $\xi$}.

Bounding transcendence measures of different constants is a classical problem in number theory. It is widely investigated in particular in the context of the Napier's constant $e$. In 1873, Hermite proved it to be transcendental \cite{hermite1873}. This work started with Borel in 1899 \cite{borel1899} when proved that $\omega(k,H)<c\log\log H$ for some constant $c$ depending on $k$. Popken \cite{popken1928,popken1929} improved the bound to $k+\frac{c}{\log\log H}$, where $c$ depends on $k$ in 1928--1929. Mahler \cite{mahler1931} made the dependance on $k$ explicit in 1931 by deriving the bound $k+\frac{ck^2\log(k+1)}{\log\log H}$, where $c$ is an absolute constant. This result is already of the shape of modern state of art results. In 1991, Khassa and Srinivasan \cite{khassasrinivasan1991} showed that $c=98$ is valid for $\log\log Hd(k+1)^{6k}$ for some constant $d>e^{950}$. This result was improved by Hata in 1995 \cite{hata1995} who showed that one can choose $c=1$, and also considerably improved the lower bound for $H$. This was further improved by Ernvall-Hyt\"onen, Matala-aho and Sepp\"al\"a ion 2018 \cite{AMLOTA}.

In this paper, we concentrate on the expression
\[\left|\lambda_ke^{k/n}+\ldots +\lambda_1e^{1/n}+\lambda_0\right|,
\]
where $k \geq n \geq 2$. This gives a transcendence measure for roots of $e$. Since $e$ is transcendental, all its roots and powers are also transcendental. Ernvall-Hyt\"onen, Matala-aho and Sepp\"al\"a considered also sparse polynomials in the context of the transcendence measure of $e$. As a corollary, they also derived a transcendence measure for integer powers of $e$. However, we were not able to find any transcendence measures tailored for \emph{roots of $e$} in the literature. There are some general results in the literature, for instance, by Mahler \cite{Mahler} which can be used to derive a bound. Also, the generalized transcendence measure by Ernvall-Hyt\"onen, Lepp\"al\"a and Matala-aho can be used to derive a bound. 
Our bound will be compared to these bounds in Section \ref{compare}.

In this article, we prove the following bound:
\begin{theorem} Assume $k\geq n\geq 2$. We have
\begin{equation}\label{ensikaava}
\left|\lambda_0+\lambda_1 e^{1/n}+\lambda_2 e^{2/n}+ \ldots +\lambda_ke^{k/n}\right| > \frac{1}{H^r},
\end{equation}
where $r>\omega (k, H)$ and we can choose
\[
\omega(k,H)=k+\frac{k^2\log k}{\log \log H}\left(1+\frac{0.69}{\log k-1}\right),
\]
for $k\geq 5$ and $\omega(k,H)=k+\frac{k^2\log k}{\log \log H}d(k)$, where
\[
d(k)=\begin{cases}
& 3.319   \textrm{ for } k=2\\
& 1.145 \textrm{ for } k=3 \\
& 1.114 \textrm{ for }k=4 
\end{cases}
\]
and $\log H\geq s(n,k)e^{s(n,k)}$ with $s(n,k)=(k+n)(\log(k+n))^2$.
\end{theorem}

We follow the approach used in \cite{AMLOTA}. 

\section{Earlier results and comparisons to our bound}
\label{compare}

The following result can be obtained as a corollary of a much more general result presented in by Mahler in 1975.

\begin{theorem}[Theorem 1 in \cite{Mahler}] Take $a_i = i$ for $i=0, \ldots, k$ (with $k\geq 2$) and $a=n$ a positive integer. Let $\lambda_0, \ldots, \lambda_k$ be integers not all zero (in Mahler's paper $x_0, \ldots, x_k$), $C(r) = (k+1)^2r \sqrt{(\log(n+k+1)\log r}$ and $T$ be the product of the non-zero $\lambda_i$. Then for $r$ the smallest integer for which \begin{align*}
    \frac{(r-1)!}{e^{2C(r-1)}} \leq \max \{ |x_i| \} < \frac{r!}{e^{2C(r)}}
\end{align*}
we have
\begin{align}
    |T(\lambda_0 + \lambda_1 e^\frac1n + \ldots + \lambda_k e^\frac{k}{n})| > \frac{\max \{ |\lambda_i| \}}{e^{(2(k+1)-\frac14) C(r)}}.
\end{align}
\end{theorem}

The following result is due to Ernvall-Hyt\"onen, Lepp\"al\"a and Matala-aho, and can be obtained as a corollary of a much more general result presented in \cite{AMKATA}.

\begin{theorem}[Corollary of \cite{AMKATA}] We have
$$\left|\lambda_0+ \lambda_1 e^{1/n} + \cdots + \lambda_k e^{k/n} \right| > \frac{M^{1 - \hat{\delta}(M)}}{h_0 h_1 \ldots h_k},$$ where $ M = \max_{ 0 \leq i \leq k } \{ |\lambda_i | \}$, $ \hat{\delta}(M) \leq \frac{\hat{B} (\overline{\alpha})}{\sqrt{\log \log M}} \leq c_k k^2 \sqrt{\log( g_1 (\overline {\alpha}) (1+ g_3(\overline {\alpha})))}/ \sqrt{\log \log M}$ and $ h_ i = \max \{ 1, | \lambda_i| \}$, for $ i = 1, \ldots, k$.  Moreover, $ c_k = 13$ if $ k  < 3 $ and $12$ otherwise. 
\end{theorem}

In particular, for  $ \overline {\alpha} = (0, \frac{1}{n}, \ldots, \frac{k}{n})$ we have $g_1 (\overline {\alpha}) = n$ and $g_3(\overline {\alpha}) = \frac{k}{n}$. 
 Therefore, 
$$\left|\lambda_0+ \lambda_1 e^{1/n} + \cdots + \lambda_k e^{k/n} \right| > \frac{M} { h_0 \cdots h_k M ^{\frac{ c_k k^2 \sqrt{\log(n (1+ k/n))}}{ \sqrt{\log \log M}}}} . $$

Let us now compare these results with our bound.

\begin{example}
Let us look at the family of polynomials with $\frac{H}{2}\leq |\lambda_i|\leq H$ for all coefficients $\lambda_i$ when $1\leq i\leq k$ and compare our result with results in \cite{AMKATA} and \cite{Mahler}.

Our result gives the bound
\[
\left|\lambda_0+ \lambda_1 e^{1/n} + \cdots + \lambda_k e^{k/n} \right|>H^{-k-\frac{k^2\log k}{\log \log H}\left(1+\frac{0.639}{\log k-1}\right)}
\]

The bound by Ernvall-Hyt\"onen, Matala-aho and Lepp\"al\"a:
\begin{align*}
    \left|\lambda_0+ \lambda_1 e^{1/n} + \cdots + \lambda_k e^{k/n} \right| &> \frac{H} { h_0 \cdots h_k H ^{\frac{ c_k k^2 \sqrt{\log(n (1+ k/n))}}{ \sqrt{\log \log H}}}}
\end{align*}

This bound is certainly not better than
\[
\left(\frac{H}2\right)^{-k-\frac{12k^2 \sqrt{\log(n (1+ k/n))}}{\sqrt{\log \log H}}}=H^{-k+k\frac{\log 2}{\log H}-\frac{12k^2 \sqrt{\log(n+ k)}}{\sqrt{\log \log H}}} ,
\] which is weaker than ours for large values of $H$ because for large $\sqrt{\log\log H}$ grows slower than $\log\log H$.

Mahler 1975:
\begin{align}
    |\lambda_0 + \lambda_1 e^\frac1n + \ldots + \lambda_k e^\frac{k}{n}| > \frac{\max \{ |\lambda_i| \}}{Te^{(2(k+1)-\frac14) C(r)}},
\end{align}
where $T$ is the product of all non-zero $\lambda_i$'s.

\begin{align}\label{mahler1}
    \frac{\max \{ |\lambda_i| \}}{Te^{(2(k+1)-\frac14) C(r)}}
   &\leq \frac{1}{\left(\frac{H}{2}\right)^{k}e^{(2(k+1)-1/4)C(r)}}.
\end{align}

Mahler gives the bound $\frac{\log x}{\log\log x}<r<\frac{6\log x}{\log\log x}$, where in his notation $x$ is the maximum of the absolute values of the coefficients of the polynomial. In our setting, this inequality would approximately translate to
\[
\frac{\log H}{\log\log H}<r<\frac{6\log H}{\log\log H}.
\]
We are losing some accuracy here because we only expected the coefficients of the polynomial to be on the interval $[\frac{H}{2},H]$. However, for the current purposes, this is not an issue.

The denominator of \eqref{mahler1} can now be written as
\[
\left(\frac{H}{2}\right)^{k}e^{(2(k+1)-1/4)(2(k+1)-1/4)C(r)}=H^{k-\frac{k\log 2-(2(k+1)-1/4)C(r)}{\log H}}=H^{k-\frac{k\log 2}{\log H}+\frac{(2(k+1)-1/4)C(r)}{\log H}}.
\]
Let us now look at the expression $\frac{(2(k+1)-1/4)C(r)}{\log H}$. Let us use the expression for $C(r)$ and the bound for $r$:
\begin{multline*}
C(r) = (k+1)^2r \sqrt{(\log(n+k+1)\log r}\approx (k+1)^2\frac{\log H}{\log \log H}\sqrt{(\log(n+k+1)\log \frac{\log H}{\log\log H}}\\ \approx (k+1)^2\frac{\log H\sqrt{\log k}}{\sqrt{\log\log H}},
\end{multline*}
where we used the bound $n\leq k$ to estimate $\log (n+k+1)\approx \log k$, and that for large H, $\log \frac{\log H}{\log \log H}\approx \log \log H$. Hence,
\[
\frac{(2(k+1)-1/4)C(r)}{\log H}\approx \frac{(2(k+1)-1/4)}{\log H}\cdot (k+1)^2\frac{\log H\sqrt{\log k}}{\sqrt{\log\log H}}\approx \frac{2(k+1)^2\sqrt{\log k}}{\sqrt{\log\log H}}.
\]
Hence, our bound is also better than Mahler's bound, because $\sqrt{\log \log H}$ grows slower than $\log\log H$, and the numerator is bigger (dependance on $k^3$ instead of $k^2$).

\end{example}

\section{Preliminaries and the outline of the method}

Ernvall-Hyt\"onen, Sepp\"al\"a and Matala-aho \cite{AMLOTA} used the following approach:

Assume that there is a sequence of simultaneous approximations
\[
L_{m,j}(h)=B_{m,0}(h)\Theta_j+B_{m,j}(h)
\]
with $m=0,1,\dots ,k$ and $j=1,2,\dots,k$. Further assume $B_{m,j}(\ell)\in \mathbb{Z}$ for all $m,j\in \{0,1,\dots,k\}$. Assume further that the coefficients $B_{m,j}$ satisfy the following determinant condition:
\[
\begin{vmatrix}
B_{0,0} & B_{0,1} & \cdots & B_{0,k}\\
B_{1,0} & B_{1,1} & \cdots & B_{1,k}\\
\vdots & \vdots & \ddots & \vdots \\
B_{k,0} & B_{m,1} & \cdots & B_{k,k}\\
\end{vmatrix} \ne 0.
\]
Pick the functions $Q(h)$, $q(h)$, $R(h)$ and $r(h)$ to be such that they satisfy the following inequalities:
\[
B_{m,0}(h)\leq Q(h) =e^{q(h)}
\]
and
\[
\sum_{j=1}^k|L_{m,j}(h)|\leq R(h)=e^{-r(h)},
\]
for all $h \geq h_0$, where the functions are of the form \[q(h)=ah\log h+bh\] and \[-r(h)=-ch\log h+dh.\]

Assume that $z(y)$ is the inverse function of the function $y(z)=z\log z$. Further, denote
\[
B=b+\frac{ad}{c},\quad C=a,\quad D=a+b+e^{-s(m)}, \quad F^{-1}=2e^D,\quad v=c-\frac{d}{s(m)},\quad h_1=\max\{h_0,e,e^{s(m)}\}.
\]
Our choice will be $s(n,k)=(n+k)(\log(n+k))^2$, and we will actually have $h_1=e^{s(n,k)}$. Under the assumptions above, they proved the following lemma:

\begin{lemma}[\cite{AMLOTA}]
Let $m\geq 1$ and $\log (2H)\geq vh_1\log h_1$. Then under the assumptions above
\begin{equation}\label{lemma_appro}
|\lambda_0+\lambda_1\Theta_1+\cdots +\lambda_m\Theta_m|>F(2H)^{-\frac{a}{c}-\epsilon(H)},
\end{equation}
where
\[
\epsilon(H)\log(2H)=Bz\left(\frac{\log(2H)}{v}\right)+C\log \left(z\left(\frac{\log (2H)}{v}\right)\right).
\]
\end{lemma}

Furthermore, they gave the following construction for the approximations in the case of $e^{\alpha_j}$: Write $\overline{\alpha} = (\alpha_0, \ldots, \alpha_k)$ and set
\begin{align}
\label{Omega_def}
\Omega(x,\overline{\alpha})=\prod_{j=0}^m(\alpha_j-x)^{\ell_j}=\sum_{i=0}^L\sigma_ix^i,
\end{align}
where $L=\ell_0+\ell_1+\cdots +\ell_m$ and $\sigma_i = \sigma_i(\overline{\ell}, \overline{\alpha})$. Then choosing
\[
A_{0}(t)=\sum_{i=\ell_0}^Lt^{L-i}i!\sigma_i,
\]
we get
\begin{equation}\label{linear_appro}
e^{\alpha_j}A_0(t)-A_{j}(t)=R_j(t),
\end{equation}
where $A_j(t)$ is a polynomial with integer coefficients and
\[
\begin{cases}
\deg A_0(t)=L-\ell_0\\
\deg A_j(t)=L-\ell_j\\
\mathrm{ord}_{t=0}R_j(t)\geq L+1.
\end{cases}
\]
Notice that the polynomials depend on the values of $\ell_0,\ell_1,\dots ,\ell_m$ and on $\overline{\alpha}$. We will explicitly describe $A_j$ and $R_j$ in the following chapter.

In the following, we will be choosing $\Theta_j=e^{j/n}$ for some $n\geq 2$. We will then proceed in the same fashion as in \cite{AMLOTA} to construct the explicit polynomials used in the simultaneous approximations and to bound them. Finally, we simplify the estimate given by \eqref{lemma_appro}.

\section{Explicit polynomial construction}

We start by constructing the simultaneous approximations of the powers of the roots of $e$.
For estimating the required term, we set $\overline{\alpha} = (\alpha_0, \ldots, \alpha_k) $ with $ \alpha_ s = s/n$, for $s =0, 1, \ldots k$. Let $\overline{\ell}=(l_0, \ldots, l_k) \in \mathbb{Z}^{k+1}_{\geq 1} $ and $L = l_0 + \ldots + \l_k$. As explained in the previous chapter, we get the following approximation formulas for $j=1,\ldots,k$ 
$$ e^{\alpha_j t} A_{0} ( t) - A_{j} ( t) = L _{j}(t),$$
 where
\[
A_{0}(t)=\sum_{i=\ell_0}^Lt^{L-i}i!\sigma_i.
\]
With a direct computation (similarly as in \cite{AMLOTA}), we obtain
\begin{align*}
\sigma_i &= (-1)^{i}\sum_{\ell_0+i_1+\ldots i_k=i}\binom{\ell_1}{i_1} \ldots \binom{\ell_k}{i_k}\left(\frac{1}{n}\right)^{\ell_1-i_1} \ldots \left(\frac{k}{n}\right)^{\ell_k-i_k} \\
&=(-1)^{i}\sum_{\ell_0+i_1+\ldots i_k=i}\binom{\ell_1}{i_1} \ldots \binom{\ell_k}{i_k}n^{-L+i}2^{\ell_2-i_2} \dots k^{\ell_k-i_k}
\end{align*}
Furthermore, $\sigma_i=0$ when $0\leq i<\ell_0$ and so
\[
A_{0}(t)= \sum_{i=0}^Lt^{L-i}i!\sigma_i.
\]
We wish to now bound the polynomials. Laplace transform gives us a tool to switch from sums to integrals, which is helpful in estimates. Since $\frac{i! \sigma_i(\overline{\ell},\overline{\alpha})}{t^{i+1}}=\mathcal{L}(\sigma_i(\overline{\ell},\overline{\alpha})x^i)(t)$ (where $\mathcal{L}$ denotes the Laplace transform), we have
\[
A_{0}(t) = \sum_{i=0}^Lt^{L-i}i!\sigma_i=t^{L+1}\sum_{i=0}^L\mathcal{L}(\sigma_i x^i)(t)=t^{L+1}\int_0^{\infty}e^{-xt}\Omega(x)dx,
\]
where $\Omega(x):=\Omega(x,\overline{\alpha}) $ is given by (\ref{Omega_def}). Now for any $\alpha_j$, we have 
\[
e^{\alpha_j t}A_{0}(t)=t^{L+1}\int_0^{\infty}e^{(\alpha_j-x)t}\Omega(x)dx=t^{L+1}\left(\int_0^{\alpha_j}+\int_{\alpha_j}^{\infty}\right)e^{(\alpha_j-x)t}\Omega(x)dx.
\]
Changing the variable in the second integral: $y=x-\alpha$ gives us:
\[
e^{\alpha_j t}A_{0}(t)=t^{L+1}\int_0^{\alpha_i}e^{(\alpha_j-x)t}\Omega(x)dx+t^{L+1}\int_{0}^{\infty}e^{-yt}\Omega(y+\alpha_j)dy
\]
Hence we get
\begin{align*}
    A_{j} (t) = t ^{L+1} \int_0 ^ \infty e ^{-yt} \Omega ( y + \alpha_j) dy
\end{align*}
and
\begin{align*}
    L_{j}(t) = t^{L+1} \int_0^{\alpha_j} e^{(\alpha_j - x)t} \Omega (x) dx
\end{align*}
for $j = 1, \ldots, k$.

We can make the result stronger if the terms in \eqref{linear_appro} are as small as possible. But at the same time, we want to keep the coefficients in $A_{j}(t)$ as integers. Therefore, we try to find as large common factors in coefficients as possible.

To do that, we proceed as in \cite{AMLOTA}. We start by picking very specific values of $\ell_0,\ell_1,\dots ,\ell_k$ in relation to each other.

For any $u$  with $0 \leq u \leq k$, we take $ \ell_s^{(u)} = \begin{cases} 
\ell - 1 & \mbox{if } s =u \\
\ell & \mbox{otherwise}  
\end{cases}$ and $\overline{\ell}^{(u)}= (\ell_0^{(u)},\ldots,\ell_k^{(u)})$. For these values of $ \overline{\ell}$, we denote  $A_j(t) = A_{ \overline{\ell}, j} (t)$ by  $ A_{ u, j} (t)$ and $L_j(t) = L_{\overline{\ell},j}(t)$ by $L_{u,j}(t)$. 

\[
A_{0}(t)=\sum_{i=\ell_0}^Lt^{L-i}i!(-1)^{i}\sum_{\ell_0+i_1+\ldots i_k=i}\binom{\ell_1}{i_1} \ldots \binom{\ell_k}{i_k}n^{-L+i}2^{\ell_2-i_2} \dots k^{\ell_k-i_k}.
\]
For our chosen $\overline{\ell}$-s we always have $\ell_0 \in \{\ell,\ell-1 \}$, so we can see that
\begin{align*}
    \frac{n^{L-\ell+1}}{(\ell-1)!}A_{u,0}(t)\in \mathbb{Z}[t].
\end{align*}
Similarly, we can look at polynomials $A_{u,j}(t)$ for $j=1,\ldots,k$.
We have
\[
A_{u,j}(t)=\int_{0}^{\infty}e^{-yt}\Omega(y+\alpha_j)dy=t^{L+1}\sum_{i=0}^{L}\mathcal{L}(\sigma_i(\overline{\ell},\overline{\beta^{(j)}}))=\sum_{i=0}^Lt^{L-i}i!\sigma_{i}(\overline{\ell},\overline{\beta^{(j)}}),
\]
where $\overline{\beta^{(j)}} = (\alpha_0 - \alpha_j, \ldots, \alpha_k - \alpha_j)  $ for each $j$. Since the coefficients $\sigma_{i}(\overline{\ell},\overline{\beta^{(j)}})$ are defined using polynomial $\Omega(\overline{\ell},\overline{\beta^{(j)}})$, and in the product representation of $\Omega$, the term $(t-\beta_j^{(j)})^{\ell_j}=(t-(\alpha_j-\alpha_j))^{\ell_j}=t^{\alpha_j}$ will define the lowest degree of terms occuring in $\Omega(\overline{\ell},\overline{\beta})$, and thereby in $A_{u,j}$, we have $\sigma_{i}(\overline{\ell},\overline{\beta})=0$ unless $i\geq \ell_j$. Hence, all the coefficients in the representation of $A_{u,j}(t)$ have the factor $\ell_j!$, which is again either $\ell$ or $\ell-1$. On the other hand, these terms have the denominator $n^{L-i}$, so again
\[
\frac{n^{L-\ell+1}}{(\ell-1)!}A_{u,j}(t)\in \mathbb{Z}[t].
\]

\section{Estimation of $ A_{u,0}(1)$}
Next, we would like to estimate the term $ A_{ u, 0} (t)$, for which its representation as an integral will be useful. We have fixed $\overline{\alpha}=(1/n,\ldots,k/n)$ and so $\Omega(x) = \prod_{j=0}^k \left(j/n - x \right)$. Furthermore, for the choices of $\overline{\ell}$ as in the previous section we have $L-1=(k+1)\ell$. Therefore, $A_{u,0}(t)$ looks like
\begin{align*}
 A_{ u, 0}(t) & = t ^{(k+1) \ell}  \int_0^\infty e^{-yt} \prod_{j=0}^k \left( \frac{j}{n} - x \right)^{\ell_s} dx \\
 & =  t ^{(k+1)\ell} \int_0^\infty e^{-xt} (-x)^\ell \left(\frac{1}{n} -x\right)^\ell \cdots \left( \frac{u}{n} -x\right) ^{ \ell -1} \cdots \left(\frac{k}{n}  -x\right)^\ell dx. \\
\end{align*}

 Note that $ \left|  \frac{ x^\ell (x - 1/n) \cdots ( x - k/n)}{(x - u/n)}\right| \leq x^{(k+1)\ell -1} \leq x^{(k+1) \ell}$, for $ x > \frac{k}{n}.$
This gives us an idea of how the function inside the integral behaves: while $0\leq x\leq \frac{k}{n}$, the function stays relatively small, and it touches zero at points $0,\frac{1}{n},\dots ,\frac{k}{n}$. However, when $x\geq \frac{k}{n}$, it starts behaving roughly as $x^{(k+1)\ell-1}e^{-x}$.

Therefore, we split the above integral in the following way 
 \begin{equation}
     \int_0^\infty e^{-x t} \prod_{j=0}^k \left( \frac{j}{n} - x \right)^{\ell_s} dx =  \left(\int_{0} ^{k/n} + \int_{k/n} ^{ 2(k+1)\ell} + \int_{2(k+1)\ell}^\infty \right) e^{-x t} \prod_{j=0}^k \left( \frac{j}{n} - x \right)^{\ell_s} dx. 
 \end{equation}
 Now we treat the above integrals with $ t =1$. In that case, 
\begin{align}
    \left| \left(\int_{k/n} ^{ 2(k+1)\ell} + \int_{2(k+1)\ell}^\infty \right) e ^{-x} \frac{ \prod_{s=0}^k \left(\frac{s}{n} -x \right) ^\ell}{ \left( \frac{u}{n} - x \right)} dx \right| \leq   \left(\int_{k/n} ^{ 2(k+1)\ell} + \int_{2(k+1)\ell}^\infty \right) e^{-x} x^{(k+1) \ell -1} dx := I_1 + I_2. 
\end{align}
For estimating the above integrals, we first consider $ I_1$. In that case, 
\begin{align*}
    |I_1| = \left|\int_{ k/n} ^{ 2(k+1) \ell} e^{-x} x^{(k+1) \ell-1} dx\right| < 2 (k+1) \ell e^{ - (k+1)\ell+1}((k+1) \ell -1)^{(k+1) \ell -1} 
\end{align*} because the expression $ e^{-x} x^{(k+1) \ell-1} $ is maximal when $  x = (k+1) \ell -1$.

Let us now move to $I_2$.
\begin{lemma}
Let $c>1$ be a constant. We have
\[
\int_{c\ell (k+1)}^{\infty} e^{-x}\frac{\left(\prod_{j=0}^{k}\left|\frac{j}{n}-x\right|\right)^{\ell}}{\left|\frac{k'}{n}-x\right|}\leq \frac{c}{c-1}e^{-c(k+1)l}(c(k+1)\ell)^{\ell (k+1)-1}.
\]
for any $0\leq k'\leq k$.
\end{lemma}

\begin{proof}
We can partially integrate:
\[
\int e^{-x}x^tdx=\left[-e^{-x}x^t\right]+\int e^{-x}tx^{t-1},
\]
which gives us the series expansion for the integral above:
\begin{multline*}
\int_{c \ell (k+1)}^{\infty}e^{-x}x^{\ell (k+1)-1}dx\\=e^{-c(k+1)\ell}(c (k+1)\ell)^{\ell (k+1)-1}+(\ell (k+1)-1)e^{-c(k+1)\ell}(c(k+1)\ell)^{\ell (k+1)-2}\\+(\ell (k+1)-1)(\ell (k+1)-2)e^{-c(k+1)\ell}(c(k+1)\ell)^{\ell (k+1)-3}+\ldots \\ \leq e^{-c(k+1)l}(c(k+1)\ell)^{\ell (k+1)-1}\left(1+\frac{(k+1)\ell-1}{c(k+1)\ell}+\frac{(k+1)\ell-1}{c(k+1)\ell}\cdot \frac{(k+1)\ell-2}{c(k+1)\ell}\ldots \right)\\ \leq \frac{c}{c-1}e^{-c(k+1)l}(c(k+1)\ell)^{\ell (k+1)-1}.
\end{multline*}
\end{proof}

Notice that if we pick $c=2$, we get the following corollary:
\begin{corollary}\label{cors2}
We have the following estimate
\[
| I_2 |\leq \int_{2\ell (k+1)}^{\infty}e^{-x}x^{\ell (k+1)-1}dx \leq 2e^{-2(k+1)l}(2(k+1)\ell)^{\ell (k+1)-1}.
\]
\end{corollary}

Next, it remains to  get a bound for 
\begin{align*}
    \int_0^{ k/n}  e ^{-y} \frac{ \prod_{s=0}^k \left(\frac{s}{n} -y \right) ^\ell}{ \left( \frac{u}{n} - y \right)} dy .
\end{align*}

\begin{lemma}
Assume $k\geq 5$. Now
\[
    \int_0^{ k/n}  e ^{-y} \frac{ \prod_{s=0}^k \left(\frac{s}{n} -y \right) ^\ell}{ \left( \frac{u}{n} - y \right)} dy \leq \frac{(k!)^\ell}{120^{\ell-1} n^{k\ell - 5 (\ell-1)}} c(n)^{\ell-1},
\]
where $c(n)=\displaystyle \max_{0\leq y \leq 1} \prod_{ s=0}^5 \left| \frac{s}{n} - y \right|^{ \ell-1}.$ Furthermore, $|c(n)|\leq 1$.
\end{lemma}

\begin{proof}

Observe that $ \displaystyle \max_{ v \leq y \leq v+1} \prod_{ s =0}^k \left| \frac{s}{n} - y \right|^{ \ell -1}  \leq \max _{ 0 \leq y \leq 1}  \prod_{ s =0} ^ k  \left| \frac{s}{n} - y \right| ^ {\ell -1}, $ for positive integer $ v $ with $ 0 \leq v \leq k -1$. Therefore,
$$ \left|  \frac{ \prod_{s=0}^k \left(\frac{s}{n} -y \right) ^\ell}{ \left( \frac{u}{n} - y \right)} \right|  \leq \frac{k!}{n^k} \max_{0 \leq y \leq 1} \prod_{ s =0} ^ k \left| \frac{s}{n} - y \right|^{ \ell -1} .$$
Hence
\begin{align*}
    \left|  \int_0^{ k/n}  e ^{-y} \frac{ \prod_{s=0}^k \left(\frac{s}{n} -y \right) ^\ell}{ \left( \frac{u}{n} - y \right)} dy \right| & \leq \int_0^{k/n} e^{-y} \frac{k!}{n^k} \max_{0 \leq y \leq 1} \prod_{ s =0} ^ k \left| \frac{s}{n} - y \right|^{ \ell -1} dy \\
    & \leq \frac{k!}{n^k} \frac{(k!)^{ \ell-1}}{(5!)^{\ell-1} n^{(k-5)(\ell-1)}} \int_0^{k/n} e^{-y} \max_{0\leq y \leq 1} \prod_{ s=0}^5 \left| \frac{s}{n} - y \right|^{ \ell-1} dy \\
    & \leq \frac{(k!)^\ell}{120^{\ell-1} n^{k\ell - 5 (\ell-1)}} c(n)^{\ell-1}, \mbox{ writing } c(n) =: \displaystyle \max_{0\leq y \leq 1} \prod_{ s=0}^5 \left| \frac{s}{n} - y \right|.
\end{align*}
By checking small values individually, and bounding
\[
\prod_{ s=0}^5 \left| \frac{s}{n} - y \right|\leq 1
\]
for $n\geq 5$, we obtain $|c(n)|\leq 1$.
\end{proof}
\begin{lemma}
Assume $k\leq 5$. Now
\[
    \int_0^{ k/n}  e ^{-y} \frac{ \prod_{s=0}^k \left(\frac{s}{n} -y \right) ^\ell}{ \left( \frac{u}{n} - y \right)} dy \leq \frac{k!}{n^k} c(n,k)^{\ell-1},
\]
where
\[
c(n,k)=\begin{cases}0.049 & \textrm{for }(n,k)=(2,2)\\ \frac{1}{16}&\textrm{for }(n,k)=(2,3)\\\frac{1}{81}&\textrm{for }(n,k)=(3,3)\\
0.114&\textrm{for }(n,k)=(2,4)\\
0.015&\textrm{for }(n,k)=(3,4)\\
0.004&\textrm{for }(n,k)=(4,4)

\end{cases}
\]
\end{lemma}
\begin{proof}
We simply bound
\[
\int_0^{ k/n}  e ^{-y} \frac{ \prod_{s=0}^k \left(\frac{s}{n} -y \right) ^\ell}{ \left( \frac{u}{n} - y \right)} dy\leq \frac{k!}{n^k}\max_{0\leq y\leq 1}\prod_{s=0}^k\left|y-\frac{s}{n}\right|\int_0^{k/n}e^{-y}dy,
\]
where the integral can be bounded to be at most $1$, and the individual maxima can be determined using WolframAlpha.

\end{proof}
\begin{lemma} Assume $k\geq 2$. We have
\[\int_0^{ \infty}  e ^{-y} \frac{ \prod_{s=0}^k \left(\frac{s}{n} -y \right) ^\ell}{ \left( \frac{u}{n} - y \right)} dy\leq \exp \left(\log 4 + (k+1) \ell \log 2 +(\ell(k+1)-1) \log ((k+1)\ell) - \ell(k+1)+1\right)
\]
\end{lemma}
\begin{proof} Assume first $k\geq 5$. Now taking the above three estimations into account, we obtain 
\begin{align*}
    &\left| A_{ u, 0}(1) \right| = \left| \int_0^\infty e^{-y} (-y)^\ell (\frac{1}{n} -y)^\ell \cdots ( \frac{u}{n} -y) ^{ \ell -1} \cdots (\frac{k}{n}  -y)^\ell dy \right|\\
    & \leq  \frac{(k!)^\ell}{120^{\ell-1} n^{k\ell - 5 (\ell-1)}} c(n)^{\ell-1}+ 2 (k+1) \ell e^{ - (k+1)\ell+1}((k+1) \ell -1)^{(k+1) \ell -1} +  2e^{-2(k+1)l}(2(k+1)\ell)^{\ell (k+1)-1}\\
     & \leq \frac{(k!)^\ell}{n^{k\ell - 5 (\ell-1)}} \left(\frac{c(n)}{120} \right)^{\ell-1}+  \left( 2 (k+1) \ell + 2 \cdot 2 ^{\ell(k+1)-1} \right) e^{-(k+1) \ell+1} ((k+1) \ell)^{ \ell(k+1) -1 }\\
      & \leq \frac{(k!)^\ell}{n^{k\ell - 5 (\ell-1)}} \left(\frac{c(n)}{120} \right)^{\ell-1}+    3 \cdot 2 ^{(k+1) \ell} e^{-(k+1) \ell+1} ((k+1) \ell)^{ \ell(k+1) -1 }\\
      & \leq 4 \cdot 2 ^{(k+1) \ell} e^{-(k+1) \ell+1} ((k+1) \ell)^{ \ell(k+1) -1 },  \qquad \mbox{ for } n \geq 2\\
       & \leq \exp \left(\log 4 + (k+1) \ell \log 2 +(\ell(k+1)-1) \log ((k+1)\ell) - \ell(k+1)+1\right)
\end{align*}

For $k \in \{2,3,4\}$ the only thing that changes is the first term, but because $c(n,k)<1$ for all choices of $k$ and $n$ that interest us, we have that
\[
\frac{k!}{n^k}c(n,k)^{\ell-1} \leq 2 ^{(k+1) \ell} e^{-(k+1) \ell+1} ((k+1) \ell)^{ \ell(k+1) -1 },
\]
so the final bound is also valid in for these values.
\end{proof}

Finally, we actually need to get a bound for $A_{u,0}^{\star}(1)$. The following corollary gives us the desired bound.
\begin{corollary}
For $k \geq 3$ and $\ell \geq \exp(s(n,k))$, we have
\[
\left| \frac{n^{L-\ell+1}}{(\ell-1)!} A_{u,0}(1) \right| \leq   \exp  \left( \ell k \log \ell + \ell \left[ k \log k+k \log n +0.72k + 0.000003 \right] \right)
\]
For $k=2$, we have  $ \displaystyle \left| \frac{n^{L-\ell+1}}{(\ell-1)!} A_{u,0}(1) \right| 
 \leq \exp( 2 \ell \log \ell + 3.377257 \ell + 2 \ell \log 2 ).$
\end{corollary}
\begin{proof}
Applying the previous lemma and Stirling's formula we have
\begin{align*}
    \log \left| A_{u,0}(1) \right| - \log \left( (l-1)! \right)& \leq \ell k \log \ell + \ell \left( (k+1) \log(k+1) + \log \ell  - \log(\ell-1) - k  +(k+1) \log 2 \right)\\
    & - \log \ell - \log (k+1)  + \frac{1}{2} \log( \ell-1)  + \log 4 -  \log\sqrt{2\pi}
\end{align*} 
We first deal with the case when $k \geq 3$. With this assumption, because $\log \ell \geq s(n,k) \geq s(2,3)$ we have that
\[
\log (\ell)-\log(\ell-1)=\int_{\ell-1}^{\ell} \frac{dx}{x}\leq \frac{1}{\ell-1}<0.000003.
\]
Furthermore,
\[
(k+1)\log(k+1)-k\log(k)=k\log(1+1/k)+\log(k+1)\leq 1+\log(k+1).
\]
and
\[
- k  +(k+1) \log 2 +1+\log(k+1)<0.72k.
\]
Additionally, for all $k \geq 2$
\[
\frac{- \log \ell - \log (k+1)  + \frac{1}{2} \log( \ell-1)  + \log 4 -  \log\sqrt{2\pi}  }{\ell}<0
\]
and goes to $0$ as $\ell$ grows.

Thus, multiplying by $n^{L-\ell+1}=n^{k \ell} $ we get
\[
    \left| \frac{n^{L-\ell+1}}{(\ell-1)!} A_{u,0}(1) \right| \leq   \exp  \left( \ell k \log \ell + \ell \left[ k \log k+k \log n +0.72k + 0.000003 \right] \right)
\] 
Let us now move to the case $k=2$. We have
\[
\log (\ell)-\log(\ell-1)=\int_{\ell-1}^{\ell} \frac{dx}{x}\leq \frac{1}{\ell-1}<0.00046.
\]
 For $ k =2 $  see that 
\begin{align*}
     & \log \left| A_{u,0}(1) \right| - \log \left( (l-1)! \right) \\
    & \leq ( \log 4  - 3 \ell +3\ell log(6 \ell)) -  (\ell-1)\log (\ell-1) +(\ell-1) - \frac{1}{2} \log( \ell-1)   -  \log\sqrt{2\pi}\\ 
     & \leq 3 \ell  \log \ell + \ell \left( 3 \log 6- 3+ \frac{\log 4 -1 }{\ell} - \frac{\ell \log(\ell-1)}{\ell} + 1 + \frac{\log(\ell-1)}{2 \ell} - \frac{\log \sqrt{2 \pi} }{\ell} \right)\\
     & \leq 2 \ell  \log \ell + \ell \left( 3 \log 6- 2+ \frac{\log 4/ \sqrt{2 \pi} - 1  }{\ell} +\log(\frac{\ell}{\ell-1})  + \frac{\log(\ell-1)}{2 \ell} \right)\\
     & \leq 2 \ell \log \ell + 3.377 \ell. 
     \end{align*} Therefore, in this case, $ \displaystyle   \left| \frac{n^{L-\ell+1}}{(\ell-1)!} A_{u,0}(1) \right| \leq \exp( 2 \ell \log \ell + 3.377257 \ell + 2 \ell \log 2 )$.

\end{proof}

\section{Integrals corresponding to terms $L_{u,j}^{\star}$}

Next we need to get a suitable bound on terms $ L_{u,j}^{\star}$, or more precisely their sum $\sum_{j=1}^k \left| L_{u,j}^{\star} \right|$. The following lemma will be useful 

\begin{lemma}
    Let $k \geq 3$ and $n \geq 2$. Then 
    \begin{align*}
        \max_{0<x<k/n} \left|x\left(\frac1n - x\right)\left(\frac2n - x\right) \cdot \ldots \cdot \left(\frac{k}{n} - x\right)\right| \leq \frac{k!}{6n^{k+1}}.
    \end{align*}
    If $k=2$, then we have $\max_{0<x<2/n} |x(\frac{1}{n}-x)(\frac{2}{n}-x)| \leq \frac{2}{3\sqrt{3}n^3}$.
\end{lemma}
\begin{proof}
    By doing a change of variable $y=nx$ the expression on the left becomes
    \begin{align*}
        \max_{0<x<k/n} \left|x\left(\frac1n - x\right) \cdot \ldots \cdot \left(\frac{k}{n} - x\right)\right| = \frac{1}{n^{k+1}}  \max_{0<y<k} \left| y(1-y) \cdot \ldots \cdot (k-y) \right|.
    \end{align*}
    By analyzing the function $\left| y(1-y)(2-y) \cdot \ldots \cdot (k-y) \right|$ we see that its maximum on the interval $(0,k)$ is attained for the first time already on the interval $(0,1)$.
    If $k \geq 3$ we have the following
    \begin{align*}
        \max_{0<y<k} \left| y(1-y) \ldots (k-y) \right| &\leq \max_{0<y<1} |(4-y)\ldots (k-y)| \cdot \max_{0<y<1} |y(1-y)(2-y)(3-y)| \\
        & \leq \frac{k!}{3!} \max_{0<y<1} y(1-y)(2-y)(3-y).
    \end{align*}
    By taking the derivative we can see that the function $y(y-1)(y-2)(y-3)$ achieves its maximum $1$ for $y=(3 \pm \sqrt{5})/2$, which finally implies that
    \begin{align*}
        \max_{0<x<k/n} \left|x\left(\frac1n - x\right) \cdot \ldots \cdot \left(\frac{k}{n} - x\right)\right| \leq \frac{k!}{6n^{k+1}}.
    \end{align*}
    Similarly for $k=2$ we need to analyze the function $y(y-1)(y-2)$, whose maximum $\frac{2}{3\sqrt{3}} $ is achieved for $y = 1 \pm \frac{1}{\sqrt{3}}$, from which the claim follows.
\end{proof}
\begin{lemma}
Let $k\geq 2$ and $n \geq 2$. Then 
\begin{align*}
    |L_{u,j}^* (1)| \leq  n^{L-\ell+1}\frac{(e^\frac{j}{n} - 1) (k!)^\ell}{(\ell-1)!(c(k)n^{k+1})^{\ell-1}n^k},
\end{align*}
where $c(k) = 6$ for $k \geq 3$ and $c(2)= 3\sqrt{3}$.
\end{lemma}

\begin{proof}
Let $j \in \{1,\ldots,k\}$. By the definition of $|L_{u,j}^* (1)|$ we have
\begin{align*}
    |L_{u,j}^* (1)| (\ell-1)! & =n^{L-\ell+1}e^\frac{j}{n} \int_0^\frac{j}{n} e^{-x} \frac{\prod_{r=0}^k |\frac{r}{n}-x|^\ell}{|\frac{u}{n}-x|} dx \\
    & =n^{L-\ell+1}e^\frac{j}{n} \int_0^\frac{j}{n} e^{-x} \prod_{r=0}^k |\frac{r}{n}-x|^{\ell-1} \frac{\prod_{r=0}^k |\frac{r}{n}-x|}{|\frac{u}{n}-x|} dx \\
    & \leq n^{L-\ell+1}e^\frac{j}{n} \int_0^\frac{j}{n} e^{-x} \prod_{r=0}^k |\frac{r}{n}-x|^{\ell-1} \frac{k!}{n^k} dx. \\
\end{align*}
Because $j \geq k$, we have that $\max_{0<x<j/n} \prod_{r=0}^k |\frac{r}{n}-x| \leq \max_{0<x<k/n} \prod_{r=0}^k |\frac{r}{n}-x|$ which is at most $\frac{k!}{c(k)n^{k+1}}$ due to the previous lemma. So we further have
\begin{align*}
    |L_{u,j}^* (1)| (\ell-1)! & \leq n^{L-\ell+1}e^\frac{j}{n} \int_0^\frac{j}{n} e^{-x} \left(\frac{k!}{c(k)n^{k+1}}\right)^{\ell-1} \frac{k!}{n^k} dx \\
    & \leq n^{L-\ell+1}\frac{(k!)^l}{(c(k)n^{k+1})^{\ell-1}n^k} e^\frac{j}{n} \int_0^\frac{j}{n} e^{-x} dx \\
    & \leq n^{L-\ell+1}\frac{(k!)^\ell}{(c(k)n^{k+1})^{\ell-1}n^k} (e^\frac{j}{n} -1) \\
\end{align*} 
\end{proof}
\begin{lemma}

Let $k \geq 2$ and $c(k)$ as in the previous lemma. We have
\[
\sum_{j=1}^k |L_{u,j}^*|\leq \frac{(k!)^{\ell}}{c(k)^{\ell-1}(\ell-1)!}n^{2-\ell}e^{(k+1)/n}.
\]
\end{lemma}

\begin{proof}
We have the following:
\[
\sum_{j=1}^k (e^{j/n}-1)<\sum_{j=1}^ke^{j/n}=\frac{e^{(k+1)/n}-e^{1/n}}{e^{1/n}-1}<\frac{e^{(k+1)/n}}{e^{1/n}-1}.
\]
Since
\[
e^{1/n}-1=\int_0^{1/n}e^xdx>\frac{1}{n},
\]
this can be further estimated to
\[
\sum_{j=1}^k (e^{j/n}-1)<ne^{(k+1)/n}.
\]
By summing up the above estimation for $j = 1, \dots, k$ we get
\[
\sum_{j=1}^k |L_{u,j}^*|\leq n^{L-\ell+1}\frac{(k!)^\ell}{(c(k)n^{k+1})^{\ell-1}n^k (\ell-1)!} ne^{(k+1)/n}.
\]
We can further simplify the above expression, by noticing that
\begin{align*}
    \frac{n^{L-\ell+1}n}{(n^{k+1})^{\ell-1}n^k}=n^{2-\ell},
\end{align*}
which follows from $L = k\ell-1$. This finally gives us the bound
\begin{align*}
    \sum_{j=1}^k |L_{u,j}^*|\leq \frac{(k!)^{\ell}}{c(k)^{\ell-1}(\ell-1)!}n^{2-\ell}e^{(k+1)/n}.
\end{align*}
\end{proof}
To make this bound suitable for application in $r(\ell)=\exp(R(\ell)$ we need to simplify it further.
\begin{lemma}
Let $k\geq 3$ and $\ell\geq e^{s(k,n)} = e^{(k+n)(\log(k+n))^2}$. We have
\[
\sum_{j=1}^k |L_{u,j}^*|\leq \exp\left(-\ell\log\ell+\ell(k \log k-0.81k-\log n+0.174)\right).
\]
For $k=2$ we have
\[
\sum_{j=1}^2 |L_{u,j}^*|\leq \exp\left( -\ell \log \ell -0.64 \ell \right).
\]
\end{lemma}

\begin{proof}
First let $k \geq 3$. We need to simplify the following expression:
\[
\log\left(\frac{(k!)^{\ell}}{6^{\ell-1}(\ell-1)!}n^{2-\ell} e^{(k+1)/n} \right)=\ell \log(k!)-(\ell-1)\log 6-\log (\ell-1)! +(2-\ell)\log n+\frac{k+1}{n}.
\]
We have $\log(\ell-1)!=\log \ell!-\log \ell$. Further,  we can use Stirling's formula to bound the factorials:
\[
\sqrt{2\pi \ell}\left(\frac{\ell}{e}\right)^{\ell}e^{1/(12\ell+1)}<\ell!<\sqrt{2\pi \ell}\left(\frac{\ell}{e}\right)^{\ell}e^{1/(12\ell)}.
\]
Hence
\[
\log \ell!>\frac{1}{2}\log(2\pi)+\frac{1}{2}\log \ell+\ell \log \ell-\ell+\frac{1}{12\ell+1}
\]
and similarly for $k!$:
\[
\log k!<\frac{1}{2}\log(2\pi)+\frac{1}{2}\log k+k \log k-k+\frac{1}{12k}.
\]
Hence, we have
\begin{align*}
\log\left(\frac{(k!)^{\ell}}{6^{\ell-1}(\ell-1)!}n^{2-\ell}e^{(k+1)/n} \right)
=&\ell \left(\frac{1}{2}\log(2\pi)+\frac{1}{2}\log k+k \log k-k+\frac{1}{12k}\right)-(\ell-1)\log 6\\
&-\log \ell! + \log \ell+(2-\ell)\log n+\frac{k+1}{n}\\
=&\ell \left(\frac{1}{2}\log(2\pi)+\frac{1}{2}\log k+k \log k-k+\frac{1}{12k}\right)-(\ell-1)\log 6+\log \ell\\
&-\left(\frac{1}{2}\log(2\pi)+\frac{1}{2}\log \ell+\ell \log \ell-\ell+\frac{1}{12\ell+1}\right) +(2-\ell)\log n+\frac{k+1}{n}\\
=&-\ell\log\ell+\ell(k \log k-0.81k-\log n+0.16)+\frac{1}{2}\log \ell \\
&+2\log n+\frac{k+1}{n}+0.88,
\end{align*}
because $\log 6-\frac{1}{2}\log(2\pi)-\frac{1}{12\ell+1}<0.88$ and $\frac{1}{2}\log(2\pi)-\log 6+\frac{1}{12k}+1<0.16$ and $\frac{1}{2}\log k-k<-0.81k$. We can further simplify by using the inequality
\[
\left(\frac{1}{2}\log \ell+2\log n+\frac{k+1}{n}+0.88\right)<0.00004 \ell
\] to obtain
\[
\log\left(\frac{(k!)^{\ell}}{6^{\ell-1}(\ell-1)!}n^{2-\ell}e^{(k+1)/n} \right)<-\ell\log\ell+\ell(k \log k-0.81k-\log n+0.17).
\]
Similar calculation for $k=2$ gives us explicitly:
\begin{align*}
    \log \left(\sum_{j=1}^2 |L_{u,j}^*| \right) \leq& -\ell \log \ell + \ell(2 \log 2 +\frac{1}{2}\log 2-2+\frac{1}{24}-\log(3\sqrt{3})+\frac{1}{2}\log(2 \pi)+1-\log n) \\
        & + \log(3\sqrt{3})-\frac{1}{2}\log(2 \pi)+\frac{1}{2} \log \ell-\frac{1}{12\ell+1}+2 \log n+ \frac{k+1}{n} \\
        \leq& -\ell \log \ell + \ell (0.0456 - \log n)+ 0.0035 \ell \\
        \leq& -\ell \log \ell -0.65 \ell+ 0.0035 \ell \leq -\ell \log \ell -0.64 \ell.
\end{align*}
\end{proof}

\section{Transcendence measure for $e^{1/n}$}

We are now ready to put together the bounds
\[
q(\ell)=\ell k \log \ell + \ell \left[ k \log k+k \log n +0.72k + 0.000003 \right]
\] which is true for $ k \geq 3$ and for $ k =2$ \[
q(\ell)= 2\ell  \log \ell +  \ell( 3.377257  + 2  \log n).
\]
  Estimating sum of $ L _{ u, j}^*$, for $k \geq 3$, we obtained 
\[
-r(\ell)=-\ell\log\ell+\ell(k \log k-0.81k-\log n+0.17)
\] and for $ k =2$, it is 
\[
-r(\ell)=-\ell \log \ell -0.64 \ell.
\]
Using notation from Section 3 in \cite{AMLOTA}, equations $(8)$ and $(9)$, we have
\begin{alignat*}{1}
s(n,k)&=(k+n)(\log(k+n))^2 \quad \textrm{this function is in place of $s(m)$ there}\\
a&=k\\
b&=k \log k+k\log n +0.72k + 0.000003\\
c&=1\\
d&=k \log k-0.81k-\log n+0.17.
\end{alignat*}
Now with the notation from Section 3 in \cite{AMLOTA}, equation $(10)$, we have 
\begin{alignat*}{1}
B&=b+\frac{ad}{c}=k \log k+k\log n +0.72k + 0.000003+k(k \log k-0.81k-\log n+0.17)\\
&=k \log k +0.89k + 0.000003+k^2\log k-0.81k^2\\
C&=a=k\\
D&=a+b+ae^{-s(k,n)}=k+k \log k+k\log n +0.72k + 0.000003+\frac{k}{e^{(k+n)(\log(k+n))^2}}\\
F^{-1}&=2e^{D}\\
v&=1-\frac{k \log k-0.81k-\log n+0.17}{(k+n)(\log(k+n))^2}\\
n_1&=e^{(n+k)(\log(n+k))^2}.
\end{alignat*}

Now we have
\[
|\lambda_0+\lambda_1e^{1/n}+\cdots +\lambda_ke^{k/n}|>F(2H)^{-a/c-\epsilon(H)},
\]
where
\[
\epsilon(H)=\frac{1}{\log(2H)}\left(Bz\left(\frac{\log(2H)}{v}\right)+C\log\left(z\left(\frac{\log(2H)}{v}\right)\right)\right).
\]
The term $H^{-a/c}=H^{-k}$ will form the main term, and everything else will be put together into the second term in the exponent (of $H$). This second term will be formed of the terms
\[
F2^{-a/c}(2H)^{-\epsilon(H)}
\]

For large $k$,  
we have 
$$1< | \Lambda | 2 (2H)^{\frac{a}{c}} e^{ \epsilon(H) \log(2H)+D} =  |\Lambda| H^{\frac{a}{c}+Y}= |\Lambda| H^{k+Y},$$ where 

\begin{align*}
    Y: &= \frac{1}{\log H} \left(Bz\left(\frac{\log(2H)}{v}\right)+C\log\left(z\left(\frac{\log(2H)}{v}\right)\right)+ D + (k +1)\log 2\right)\\
    & \leq \frac{1}{\log H} \left( \frac{uB}{v} \frac{\log(2H)}{\log \log(2H)} +k \log \left( \frac{u}{v} \frac{\log(2H)}{\log \log(2H)} \right) + D +(k +1)\log 2 \right)
    \end{align*}

and  $$  u = 1 + \frac{\log(s(n,k))}{s(n,k)}.$$ We use the fact that $ \log H \geq  s (n, k) e^{s(n,k)}$. Not, we need to estimate the terms involving $D$ and $B$. 
    \begin{align*}
        D+(k +1)\log 2  & =  k+k \log k+k\log n +0.72k + 0.000003+\frac{k}{e^{(k+n)(\log(k+n))^2}} +(k +1)\log 2 \\
        & = k + k \log k + k \log n +k \left( 0.72 + \frac{0.000003}{k} + \frac{1}{e^{(k+n)(\log(k+n))^2}}  + \log 2 + \frac{\log 2}{k}\right) \\
        & \leq k \log k + k \log n +3.4 k. 
    \end{align*}
     For the next term, we observe
      $$ k \log \left( \frac{u}{v} \frac{\log(2H)}{\log \log(2H)} \right) \leq  k \log( u \log(2H))  .$$
Therefore, we get 
\begin{align*}
    Y &  \leq \frac{1}{\log H} \left( \frac{uB}{v} \frac{\log(2H)}{\log \log(2H)} +k  \log( u \log(2H)) + k \log k + k \log n +3.4 k \right)\\
    & \leq \frac{1}{\log H} \left( \frac{uB}{v} \frac{\log(2H)}{\log \log(2H)} +k  \log( 2 \log H)+ k \log (2 \log 2) + k \log k + k \log n +3.4 k \right)\\
    & \leq \frac{1}{\log H} \left( \frac{uB}{v} \frac{\log(2H)}{\log \log(2H)} +k  \log( 2 \log H)+ \frac{k}{2} + k \log k + k \log n +3.4 k \right)\\
    & \leq \frac{1}{\log H} \left( \frac{uB}{v} \frac{\log(2H)}{\log \log(2H)} + \frac{69}{10} k \log(2 \log H) \right)\\
    & \leq \frac{1}{ \log \log H} \left( \frac{\log(2H)}{\log H} \cdot \frac{uB}{v} +  \frac{ 6. 9 \log \log H  \cdot k  \log( 2 \log H)}{\log H} \right) \\
    & = \frac{u}{v \log \log H} \left( B + \frac{1}{\log H} \left(\log(2) B +  \frac{ v\log \log H  \cdot 6. 9 k  \log( 2 \log H)}{u} \right) \right)\\
  \end{align*}
  We have
 \[ 
Y   \leq \begin{cases}
& \frac{u}{v \log \log H} \left( B +0.744754115 \right), \mbox{ for } k =2 \\
& \frac{u}{v \log \log H} \left( B + 0.04386773 \right), \mbox{ for } k =3\\
& \frac{u}{v \log \log H} \left( B + 0.00075786 \right), \mbox{ for } k =4\\
 & \frac{u}{v \log \log H} \left( B + 0.00000412 \right), \mbox{ for } k =5\\ 
& \frac{u}{v \log \log H} \left( B +  7. 976 \times 10^{-9} \right), \mbox{ for } k =6\\
 \end{cases}
 \]
   For $ k \geq 6,$ we observe
   $$ Y \leq \frac{u}{v \log \log H} \left( B +   10^{-8} \right).$$ 
\smallskip
   
For calculating the small values of $k$, we do it case by case. \\

For $k =2$,   take $n=2$ and recall
\begin{align*}
    & b = 3.377257  + 2  \log 2\\
    & d = -0.64\\
    & a =2 \mbox{ and } d =1
\end{align*} which implies $$ 3.878864 \mbox { and } D = 7.159781 \qquad \frac{u}{v} \leq 1.151906$$

We also have $ \log H \geq s(2,2) e^{s(2,2)} = 7.69 \times e^{7.69}$

\bigskip
\begin{align*}
    Y    & \leq \frac{1}{\log H} \left( \frac{uB}{v} \frac{\log(2H)}{\log \log(2H)} +2 \log \left( \frac{u}{v} \frac{\log(2H)}{\log \log(2H)} \right) + D +3\log 2 \right)\\
     &  \leq \frac{ 9.202255}{\log \log H}
\end{align*}

Let us now look at the case $k\geq 3$. For simplicity's sake, write
\[
Y\leq \frac{u}{v\log\log H}(B+\theta)
\]

Define
\[
f(k,n)=\frac{u}{v k^2\log k} \left( B +   \theta \right)=\frac{\left(1+\frac{\log((k+n)(\log(k+n))^2)}{(k+n)(\log(k+n))^2}\right)\left(1+\frac{1}{k}+\frac{0.89}{k\log k}+\frac{0.000003+\theta}{k^2\log k}-\frac{0.81}{\log k}\right)}{1-\frac{k\log k-0.81k-\log n+0.17}{(k+n)(\log(k+n))^2}}
\]
The expression can be further simplified to
\[
f(k,n)=\frac{\left(1+\frac{1}{(k+n)\log(k+n)}+\frac{2\log\log(k+n)}{(k+n)(\log(k+n))^2}\right)\left(1+\frac{1}{k}+\frac{0.89}{k\log k}+\frac{0.000003+\theta}{k^2\log k}-\frac{0.81}{\log k}\right)}{1-\frac{k\log k-0.81k-\log n+0.17}{(k+n)(\log(k+n))^2}}
\]
Let us look at the numerator. We have
\begin{multline*}
    \left(1+\frac{1}{(k+n)\log(k+n)}+\frac{2\log\log(k+n)}{(k+n)(\log(k+n))^2}\right)\left(1+\frac{1}{k}+\frac{0.89}{k\log k}+\frac{0.000003+\theta}{k^2\log k}-\frac{0.81}{\log k}\right)\\=1+\frac{1}{k}+\frac{0.89}{k\log k}+\frac{0.000003+\theta}{k^2\log k}
-\frac{0.81k}{\log k}+\frac{1}{(k+n)\log(k+n)}+\frac{1}{k(k+n)\log(k+n)}+\frac{0.89}{k(k+n)\log k\log(k+n)}\\+\frac{0.000003+\theta}{k^2\log k(k+n)\log(k+n)}-\frac{0.81}{(k+n)\log(k+n)\log k}+\frac{2\log \log (k+n)}{(k+n)(\log(k+n))^2}+\frac{2\log\log(k+n)}{k(k+n)(\log(k+n))^2}\\+\frac{2\cdot 0.89\log\log(k+n)}{k(k+n)\log k(\log(k+n))^2}+\frac{2\cdot (0.000003+\theta)\log\log(k+n)}{k^2(k+n)\log k(\log(k+n))^2}-\frac{2\cdot 0.81\log\log (k+n)}{(k+n)(\log(k+n))^2\log k}.\end{multline*}
We can now verify using WolframAlpha that
\[
    \frac{1}{k(k+n)\log(k+n)}+\frac{0.89}{k(k+n)\log k\log(k+n)}+\frac{0.000003+\theta}{k^2\log k(k+n)\log(k+n)}-\frac{0.81}{(k+n)\log(k+n)\log k}<0
\]
and
\begin{multline*}
    \frac{2\log\log(k+n)}{k(k+n)(\log(k+n))^2}\\+\frac{2\cdot 0.89\log\log(k+n)}{k(k+n)\log k(\log(k+n))^2}+\frac{2\cdot (0.000003+\theta)\log\log(k+n)}{k^2(k+n)\log k(\log(k+n))^2}-\frac{2\cdot 0.81\log\log (k+n)}{(k+n)(\log(k+n))^2\log k}<0.
\end{multline*}
for $k\geq 3$. Furthermore, the denominator can be written as
\begin{multline*}
1-\frac{k\log k}{(k+n)(\log(k+n))^2}+\frac{0.81k}{(k+n)(\log(k+n))^2}+\frac{\log n}{(k+n)(\log (k+n))^2}-\frac{0.17}{(k+n)(\log(k+n))^2}\\ >1-\frac{k\log k}{(k+n)(\log(k+n))^2}+\frac{0.81k}{(k+n)(\log(k+n))^2},
\end{multline*}
since $\frac{\log n}{(k+n)(\log (k+n))^2}-\frac{0.17}{(k+n)(\log(k+n))^2}>0$.
We can thus estimate
\[
f(n,k)<\frac{ 1+\frac{1}{k}+\frac{0.89}{k\log k}+\frac{0.000003+\theta}{k^2\log k}
-\frac{0.81}{\log k}+\frac{1}{(k+n)\log(k+n)}+\frac{2\log \log (k+n)}{(k+n)(\log(k+n))^2}}{1-\frac{k\log k}{(k+n)(\log(k+n))^2}+\frac{0.81k}{(k+n)(\log(k+n))^2}}.
\]
If $k=3$, we have $f(2,3)<1.145$ and $f(3,3)<1.08$, so $f(2,3)$ gives the larger value.

For $k=4$, we have $f(2,4)<1.114$, $f(3,4)<1.05$ and $f(4,4)<1$, so $f(2,4)$ yields the largest bound.
Estimating
\[
\frac{1}{(k+n)\log(k+n)}+\frac{2\log \log (k+n)}{(k+n)(\log(k+n))^2}<\frac{0.81k}{(k+n)(\log(k+n))^2}
\]
(which is based on the fact that the second term on the left side is at most $1+\frac{2}{e}<1.74$ and the term on the right side is at least $>1.75$ if $k\geq 5$ and $n\leq k$), and
\[
\frac{k\log k}{(k+n)(\log(k+n))^2}<\frac{1}{\log k},
\]
we can further simplify the expression:
\[
f(n,k)<\frac{ 1+\frac{1}{k}+\frac{0.89}{k\log k}+\frac{0.000003+\theta}{k^2\log k}
-\frac{0.81}{\log k}+\frac{0.81k}{(k+n)(\log(k+n))^2}}{1-\frac{1}{\log k}+\frac{0.81k}{(k+n)(\log(k+n))^2}}
\]
Now
\begin{multline*}
\frac{1}{k}+\frac{0.89}{k\log k}+\frac{0.000003+\theta}{k^2\log k}=\frac{1}{\log k}\left(\frac{\log k}{k}+\frac{0.89}{k}+\frac{0.000003+\theta}{k^2\log k}\right)\\ <\frac{1}{\log k}\left(0.299+0.149+0.0000003\right)<\frac{0.5}{\log k},
\end{multline*}
where the terms are estimated using $k\geq 5$. Hence
\[
f(n,k)<\frac{ 1
-\frac{0.31}{\log k}+\frac{0.81k}{(k+n)(\log(k+n))^2}}{1-\frac{1}{\log k}+\frac{0.81k}{(k+n)(\log(k+n))^2}}=1+\frac{0.69}{\log k-1+\frac{0.81k\log k}{(k+n)(\log(k+n))^2}}<1+\frac{0.69}{\log k-1}
\]
Hence,
\[
Y\leq \frac{k^2\log k}{\log \log H}f(n,k)<\frac{k^2\log k}{\log \log H}\left(1+\frac{0.69}{\log k-1}\right).
\].

%

\end{document}